\documentclass{amsart}

\usepackage[foot]{amsaddr}
\usepackage{lmodern}
\usepackage[T1]{fontenc}
\usepackage[utf8]{inputenc}

\usepackage{overpic}
\usepackage{todonotes}

\usepackage{amsmath,amssymb,amsfonts}
\usepackage[sort&compress,square,numbers]{natbib}
\usepackage{mathtools}

\usepackage{hyperref}
\hypersetup{
  pdfauthor={Daniel F. Scharler, Hans-Peter Schröcker},
  pdftitle={An Algorithm for the Factorization of Split Quaternion Polynomials},
  pdfkeywords={},
  hidelinks,
}

\usepackage{algorithm}
\usepackage{algpseudocode}

\algrenewcommand\algorithmicwhile{\textbf{While}}
\algrenewcommand\algorithmicfor{\textbf{For}}
\algrenewcommand\algorithmicdo{\textbf{Do}}
\algrenewcommand\algorithmicif{\textbf{If}}
\algrenewcommand\algorithmicthen{\textbf{Then}}
\algrenewcommand\algorithmicelse{\textbf{Else}}
\algrenewcommand\algorithmicend{\textbf{End}}
\algrenewcommand\algorithmicreturn{\textbf{Return}}

\DeclareMathOperator{\quo}{quo}
\DeclareMathOperator{\rem}{rem}
\DeclareMathOperator{\lquo}{lquo}
\DeclareMathOperator{\lrem}{lrem}
\DeclareMathOperator{\rquo}{rquo}
\DeclareMathOperator{\rrem}{rrem}

\newcommand{\ci}{\mathrm{i}}
\newcommand{\qi}{\mathbf{i}}
\newcommand{\qj}{\mathbf{j}}
\newcommand{\qk}{\mathbf{k}}
\newcommand{\Cj}[1]{{#1}^\ast}
\newcommand{\SPQ}{\mathbb{S}}
\newcommand{\R}{\mathbb{R}}
\newcommand{\C}{\mathbb{C}}
\renewcommand{\H}{\mathbb{H}}
\renewcommand{\DH}{\mathbb{DH}}
\newcommand{\eps}{\varepsilon}
\newcommand{\SO}[1][3]{\mathrm{SO}(#1)}
\newcommand{\SE}[1][3]{\mathrm{SE}(#1)}
\newcommand{\N}{\mathcal{N}}
\newcommand{\NC}{\mathfrak{N}}
\renewcommand{\P}{\mathbb{P}}
\DeclareMathOperator{\RE}{Re}
\DeclareMathOperator{\IM}{Im}
\newcommand{\Scalar}[1]{\RE(#1)}
\newcommand{\Vector}[1]{\IM(#1)}

\newtheorem{thm}{Theorem}[section]
\newtheorem{lem}[thm]{Lemma}

\newtheorem{cor}[thm]{Corollary}
\theoremstyle{definition}
\newtheorem{defn}[thm]{Definition}
\theoremstyle{remark}
\newtheorem{rmk}[thm]{Remark}
\newtheorem{example}[thm]{Example}

\title{An Algorithm for the Factorization of Split Quaternion Polynomials}

\date{\today}

\author{Daniel F. Scharler \and Hans-Peter Schröcker}
\address{Department of Basic Sciences in Engineering Sciences, University of Innsbruck, Technikerstr.~13, 6020 Innsbruck, Austria}
\email{daniel.scharler@uibk.ac.at}
\email{hans-peter.schroecker@uibk.ac.at}


\begin{document}

\begin{abstract}
  We present an algorithm to compute all factorizations into linear factors of
  univariate polynomials over the split quaternions, provided such a
  factorization exists. Failure of the algorithm is equivalent to
  non-factorizability for which we present also geometric interpretations in
  terms of rulings on the quadric of non-invertible split quaternions. However,
  suitable real polynomial multiples of split quaternion polynomials can still
  be factorized and we describe how to find these real polynomials. Split
  quaternion polynomials describe rational motions in the hyperbolic plane.
  Factorization with linear factors corresponds to the decomposition of the
  rational motion into hyperbolic rotations. Since multiplication with a real
  polynomial does not change the motion, this decomposition is always possible.
  Some of our ideas can be transferred to the factorization theory of motion
  polynomials. These are polynomials over the dual quaternions with real norm
  polynomial and they describe rational motions in Euclidean kinematics. We
  transfer techniques developed for split quaternions to compute new
  factorizations of certain dual quaternion polynomials.
\end{abstract}

\maketitle

\section{Introduction}

In kinematics, robotics and mechanism science Hamiltonian quaternions and dual
quaternions have been employed to parametrize the group of Euclidean
displacements $\SE[3]$ as well as its subgroups $\SO$ and $\SE[2]$. The rich
algebraic structure of the quaternion models allow to investigate certain
problems from an algebraic point of view. Rational motions can be represented by
polynomials over the ring of quaternions or dual quaternions. In this context,
factorization of a polynomial into polynomials of lower degree corresponds to
the decomposition of a rational motion into ``simpler'' motions. One of the
simplest non-trivial motions are rotations. They can be represented by linear
polynomials. On the other hand, linear polynomials generically represent
rotational motions. Hence, a motion described by a polynomial that admits a
factorization into linear factors can be realized by a mechanism whose revolute
joints correspond to the linear factors.

A suitable model for motions in the hyperbolic plane is provided by the
non-commutative ring of split quaternions \cite[Chapter~8]{alkhaldi20}. In
contrast to the (Hamiltonian) quaternions, the presence of zero divisors makes
the factorization theory of polynomials over split quaternions more complex.
Factorization of quadratic split quaternion polynomials has been investigated in
\cite{abrate09} and \cite{cao19}. Zeros of split quaternion polynomials of
higher degree (which are closely related to linear factors, c.f.
Lemma~\ref{lem:zero-factor}) are the topic of \cite{falcao19}. Based on the
theory of motion factorization by means of quaternions \cite{niven41,gordon65}
and dual quaternions \cite{hegedus13,li15b} a characterization of
factorizability for such polynomials has been found \cite{scharler19a}. In order
to compute factorizations, the algorithm of the Euclidean setup has been adapted
\cite{scharler19a, li18d}.

In this article, we consider rational motions in the hyperbolic plane,
represented by polynomials over the split quaternions. We extend the results
from the quadratic case \cite{scharler19a} to polynomials of arbitrary degree. A
main ingredient is a geometric interpretation of factorizability. We investigate
the ``geometry'' of the factorization algorithm for generic cases and modify it
such that it provably finds all factorizations into linear factors. Special
cases include polynomials with infinitely many factorizations or no
factorizations at all.

In case a split quaternion polynomial representing a rational motion has no
factorization with linear factors we can adopt a ``degree elevation technique''
from the Euclidean setup \cite{li15b}: Multiplying with a suitable real
polynomial does not change the underlying rational motion but allows the
decomposition into linear factors. Hence, the initial motion can be decomposed
into hyperbolic rotations. In contrast to the Euclidean case, our approach is
rather based on geometric considerations than algebraic ones.

It is worth mentioning that the set of split quaternions is isomorphic to the
fundamental algebra of real $2 \times 2$ matrices and all of our results can be
directly transferred.

The further structure of this article is as follows: After a brief definition of
split quaternions and hyperbolic motions in Section~\ref{sec:preliminaries} we
recall the state of art in factorization theory in
Section~\ref{sec:state-of-the-art}. In Section~\ref{sec:factor-non-generic} we
introduce a new algorithm for factorization of split quaternion polynomials. In
Section~\ref{sec:geometry-of-algorithm} we investigate the geometry of the
algorithm and present a degree-elevation technique for dealing with
non-factorizable polynomials as well. Finally, we apply some of our ideas to
Euclidean motions in Section~\ref{sec:euclidean-factorization} and compute new
factorization for dual quaternion polynomials.

\section{Preliminaries}
\label{sec:preliminaries}

\subsection{Split Quaternions and their Geometry}
\label{sec:split-quaternions-and-geometry}

The algebra of split quaternions $\SPQ$ is defined as the set
\begin{align*}
  \{ h = h_0 + h_1 \qi + h_2 \qj + h_3 \qk \colon h_0, h_1, h_2, h_3 \in \R \}.
\end{align*}
Addition of split quaternions is done component-wise and the non-commutative
multiplication is obtained from the relations
\begin{align*}
  \qi^2 = - \qj^2 = - \qk^2 = -\qi \qj \qk = -1.
\end{align*}
The \emph{conjugate} split quaternion of $h = h_0 + h_1 \qi + h_2 \qj + h_3 \qk$
is defined by $\Cj{h} \coloneqq h_0 - h_1 \qi - h_2 \qj - h_3 \qk$ and the
expression
\begin{align*}
  h \Cj{h} = \Cj{h} h = h_0^2 + h_1^2 - h_2^2 - h_3^2 \in \R
\end{align*}
is sometimes called \emph{norm} in the context of quaternions, even if it can
attain negative values and lacks the square root. The center of $\SPQ$ is
precisely the field of real numbers $\R$, i.e. $hr = rh$ for all $r \in \R$.
Later on we will also allow complex coefficients $h_0$, $h_1$, $h_2$, $h_3 \in
\C$ and multiply by complex numbers $z \in \C$ that commute with split
quaternions as well. For the complex unit we will use the (non-boldface) symbol
$\ci \in \C$ with the property $\ci^2=-1$. Conjugation of complex numbers is
denoted by $\overline{z} \in \C$. Let $g \in \SPQ$ be another split quaternion,
then $\Cj{(hg)} = \Cj{g}\Cj{h}$ and $(hg) \Cj{(hg)} = h\Cj{h} g\Cj{g}$. The
inverse of $h$ is $h^{-1} = (h \Cj{h})^{-1} \Cj{h}$. It exists if and only if
$h\Cj{h} \neq 0$. We call $\Scalar{h} \coloneqq \frac{1}{2}(h + \Cj{h}) = h_0$
the \emph{scalar part} and $\Vector{h} \coloneqq \frac{1}{2}(h - \Cj{h}) =
h_1\qi + h_2\qj + h_3\qk$ the \emph{vector part} of $h \in \SPQ$. If $\Scalar{h}
= 0$ then $h$ is called \emph{vectorial}. The \emph{scalar product} and the
\emph{cross product} of $h$ and $g$ are defined by
\begin{align*}
  \langle h,g \rangle \coloneqq \frac{1}{2}(h\Cj{g} + g\Cj{h}) \quad \text{and} \quad h \times g \coloneqq \frac{1}{2}(hg - gh),
\end{align*}
respectively. Note that the cross product is not affected by the scalar parts,
i.e. $h \times g = \Vector{h} \times \Vector{g}$. The real four-dimensional
vector space of split quaternions together with the symmetric bilinear form
\begin{equation*}
  \langle \cdot, \cdot \rangle \colon \SPQ \times \SPQ \to \R \colon (h,g) \mapsto \langle h,g \rangle
\end{equation*}
is a \emph{pseudo-Euclidean space}. The set
\begin{align*}
  \NC \coloneqq \{ h \in \SPQ \colon \langle h,h \rangle = 0 \}
\end{align*}
is called the \emph{null cone}. It consists precisely of all zero divisors in
$\SPQ$. We will also consider the projective space $\P(\SPQ)$ over $\SPQ$. An
element of $\P(\SPQ)$ is denoted by $[h]$ where $h \in \SPQ \setminus \{ 0 \}$
is a non-zero split quaternion. Let $[g],[h] \in \P(\SPQ)$ be two elements. The
projective span of $[h]$ and $[g]$ is denoted by
\begin{align*}
  [h] \vee [g] \coloneqq \{ [\lambda h + \mu g] \in \P(\SPQ) \colon [\lambda, \mu]^\intercal \in \P(\R) \}. 
\end{align*}
It is a straight line in general and only a point if $[h] = [g]$, i.e. $h$ and
$g$ are linearly dependent in the vector space~$\SPQ$.

\begin{defn}[\cite{scharler19a}]
  \label{def:null_quadric}
  The quadric $\N$ in $\P(\SPQ)$ represented by the symmetric bilinear form
  $\langle \cdot, \cdot \rangle$ is called the \emph{null quadric}. Lines
  contained in $\N$ are called \emph{null lines} or \emph{rulings} of $\N$.
\end{defn}

The null quadric $\N$ is of hyperbolic type because the signature of $\langle
\cdot,\cdot \rangle$ is $(2,2)$. Moreover, $\N$ carries two families of lines.
In the following we recall some basic results on null lines and the null quadric
$\N$. For the proofs we refer to \cite{li18a} and~\cite{scharler19a}.

\begin{thm}[\cite{scharler19a}]
  \label{th:left-right-ruling}
  If $[h]$ is a point of $\N$, the sets
  \begin{equation*}
    \mathcal{L} \coloneqq \{ [g] \in \P(\SPQ) \colon g\Cj{h} = 0 \}
    \quad\text{and}\quad
    \mathcal{R} \coloneqq \{ [g] \in \P(\SPQ) \colon \Cj{h}g = 0 \}
  \end{equation*}
  are the two different rulings of $\N$ through~$[h]$.
\end{thm}

\begin{defn}
  \label{def:left-right-ruling}
  Consider a ruling $L \subset \N$. If $g\Cj{h} = 0$ for all $[h]$, $[g] \in L$,
  then $L$ is called a \emph{left ruling}. If $\Cj{h}g = 0$ for all $[h]$, $[g]
  \in L$, then $L$ is called a \emph{right ruling}.
\end{defn}

\begin{cor}[\cite{scharler19a}]
  \label{cor:left-right-ruling}
  Consider two points $[h] \in \N$ and $[g] \in \P(\SPQ)$.
  \begin{itemize}
  \item If $g$ is such that $gh \neq 0$ and $[gh] \neq [h]$, then $[h] \vee
    [gh]$ is a left ruling of~$\N$.
  \item If $g$ is such that $hg \neq 0$ and $[hg] \neq [h]$, then $[h] \vee
    [hg]$ is a right ruling of~$\N$.
  \end{itemize}
\end{cor}

\begin{rmk}\label{rmk:clifford-translation}
  For fixed $g \in \SPQ \setminus \NC$ the maps $[x] \mapsto [gx]$ and $[x]
  \mapsto [xg]$ are the well-known Clifford left and right translations of
  non-Euclidean geometry (c.\,f.~\cite{scharler19a}). If $g \in \NC \setminus
  \{0\}$, these maps are still defined on $\P(\SPQ)$ minus the subspaces of
  left/right annihilators of $g$. We refer to them as \emph{singular} Clifford
  left or right translations. Their image is the left/right ruling through
  $[g]$, respectively.
\end{rmk}

\begin{thm}[\cite{scharler19a}]
  \label{th:affine-two-plane-of-zeros}
  Consider two split quaternions $h = h_0 + h_1 \qi + h_2 \qj + h_3 \qk \in \SPQ
  \setminus \{ 0 \}$ and $g = g_0 + g_1 \qi + g_2 \qj + g_3 \qk \in \SPQ
  \setminus \{ 0 \}$ such that $[h] \vee [g]$ is a left ruling of $\N$. There
  exists an affine two-plane consisting of all split quaternions $x \in \SPQ$
  solving the equation $g = xh$. It can be parameterized by $u + \lambda \Cj{h}
  + \mu \qi \Cj{h}$, where $u = (g_0 + g_1 \qi) (h_0 + h_1 \qi)^{-1}$ and
  $\lambda$, $\mu \in \R$. (If $[h] \vee [g]$ is a right ruling, the same
  statement holds for $g = hx$ with $u + \lambda \Cj{h} + \mu \Cj{h} \qi$ where
  $u = (h_0 + h_1 \qi)^{-1} (g_0 + g_1 \qi)$ and $\lambda$, $\mu \in \R$. It
  even holds true if $[h]$ and $[g]$ coincide and therefore do not span a
  straight line.)
\end{thm}

\subsection{Hyperbolic Motions and Split Quaternion Polynomials}
\label{sec:hyperbolic-motions-and-polynomials}

In this section we recall the split quaternion model for hyperbolic motions
of \cite[Chapter~8]{alkhaldi20} employing the terminology of \cite{wildberger10,
  wildberger11, wildberger13}. Consider the three-dimensional vector space
\begin{align*}
  \Vector{\SPQ} \coloneqq \{ h \in \SPQ \colon \Scalar{h} = 0 \}
\end{align*}
of all vectorial split quaternions and the projective plane $H^2$ over
$\Vector{\SPQ}$. Points of $H^2$ are again denoted by $[h]$ where $h \in
\Vector{\SPQ} \setminus \{ 0 \}$. The null quadric $\N$ intersects $H^2$ in the
\emph{absolute circle} or \emph{null circle}.

A \emph{rotation} in the hyperbolic plane $H^2$ is defined as the composition of
two reflections. A reflection is a homology of $H^2$ in the sense of
\cite[Section~5.7]{casas-alvero14} (a projective collineation, different from
the identity, with a line of fixed points and a bundle of fixed lines) that
preserves the absolute circle $\N \cap H^2$. The intersection point of all lines
in the bundle of fixed lines is called the \emph{center} of the reflection and
the line of fixed points is called its \emph{axis.} Center and axis are polar to
each other with respect to the null circle. A rotation has a fix point as
rotation center as well, the intersection point of the two reflection axes. In
our setting it makes sense to consider reflections as special (idempotent)
rotations because they are described by the same formula \eqref{eq:rotation} in
Theorem~\ref{th:rotation} below. Formally, this is accomplished by considering
the identity as a reflection with unspecified center and axis.

The split quaternion approach allows to describe
rotation in terms of split quaternion multiplication.

\begin{thm}[{\cite[Chapter~8]{alkhaldi20}}]\label{th:rotation}
  Let $h \in \SPQ$ be an invertible split quaternion. Then the map
  \begin{align}\label{eq:rotation}
    \varrho \colon H^2 \to H^2 \colon [x] \mapsto [h x \Cj{h}]
  \end{align}
  is a rotation with rotation center $[\Vector{h}] \in H^2$. It is a reflection
  if $\RE(h) = 0$.
\end{thm}

Rotations preserve distances with respect to the scalar product $\langle \cdot,
\cdot \rangle$. Consider an invertible split quaternion $h \in \SPQ$ and two
points $[v]$, $[w] \in H^2$ not contained in $\N$. The \emph{quadrance} (squared
distance) of $[v]$ and $[w]$ is defined by
\begin{align*}
  Q([v],[w]) \coloneqq 1 - \frac{\langle v, w \rangle^2}{\langle v, v \rangle \langle w, w \rangle}.
\end{align*}
It is a concept of universal hyperbolic geometry in the
sense of N.~Wildberger \cite{wildberger13}. In contrast to traditional hyperbolic geometry, the
quadrance might be zero and even attain negative values. Because of
\begin{multline*}
  2\langle hv\Cj{h}, hw\Cj{h} \rangle =
  hv\Cj{h}h\Cj{w}\Cj{h} +
  hw\Cj{h}h\Cj{v}\Cj{h}
  = (h\Cj{h})h(v\Cj{w}+w\Cj{v})\Cj{h} = 2(h\Cj{h})^2 \langle  v, w \rangle
\end{multline*}
it is unchanged when displacing $[v]$ and $[w]$ via \eqref{eq:rotation}. Thus,
rotations indeed preserve distances in $H^2$. For more details we refer to
\cite[Chapter~8]{alkhaldi20}.

The ring of polynomials in one indeterminate $t$ (where multiplication is
defined by the convention that $t$ commutes with all coefficients) is denoted by
$\SPQ[t]$. This convention on multiplication is motivated by our kinematic
interpretation of parametrizing motions in the hyperbolic plane via split
quaternion polynomials. The parameter $t$ can be viewed as a real motion
parameter. In the same spirit, we also define $\Cj{t} \coloneqq t$.

Consider a split quaternion polynomial $P = \sum_{\ell=0}^n p_\ell t^\ell \in
\SPQ[t]$ of degree $n \in \mathbb{N}$ with $p_0,\ldots ,p_n \in \SPQ$ and $p_n
\neq 0$. The \emph{conjugate polynomial} of $P$ is $\sum_{\ell=0}^n \Cj{p}_\ell
t^\ell$ and its \emph{norm polynomial} $P \Cj{P} = \Cj{P} P$ is a real
polynomial. \emph{Right evaluation} of $P$ at $h \in \SPQ$ is defined by $P(h)
\coloneqq \sum_{\ell = 0}^n p_\ell h ^\ell$, i.e. the indeterminate $t$ is only
substituted by $h$ after it is written to the right hand side of the
coefficients in expanded form. This is important when evaluating a product of
polynomials. Consider the two polynomials $t$, $\qi \in \SPQ[t]$. Their product
evaluated at $\qj \in \SPQ$ is equal to $(t\qi) (\qj) = (\qi t)(\qj) = \qi \qj =
\qk \neq -\qk = \qj \qi$. A split quaternion $h \in \SPQ$ is called a
\emph{right zero} of $P \in \SPQ[t]$ if $P(h) = 0$. \emph{Left evaluation} and
\emph{left zeros} are defined in the same manner.

\begin{defn}\label{def:factorization}
  A split quaternion polynomial $P \in \SPQ[t]$ of degree $n \geq 1$ admits a
  factorization into linear factors if there exist split quaternions $p$,
  $h_1,\ldots, h_n \in \SPQ$ such that
  \begin{align*}
    P = p (t-h_1) \cdots (t-h_n).
  \end{align*}
\end{defn}

The rightmost factor $t-h_n$ in Definition~\ref{def:factorization} is linked to
a right zero of $P$. We recall this result from \cite{li18d}.

\begin{lem}
  \label{lem:zero-factor}
  The split quaternion $h \in \SPQ$ is a left/right zero of the polynomial $P \in
  \SPQ[t]$ if and only if $t - h$ is a left/right factor of $P$.
\end{lem}

In contrast to polynomials over the real or complex numbers, this statement does
not hold for all linear factors of a split quaternion polynomial. Consider the
polynomial $P = (t-\qj)(t-\qi)$. Evaluation of $P$ at $\qi$ and $\qj$ yields
\begin{align*}
  P(\qi) &= (t-\qj)(t-\qi)(\qi) = (t^2 - (\qi + \qj)t - \qk)(\qi) = \qi^2 - \qi^2 - \qj\qi - \qk = 0, \\
  P(\qj) &= \qj^2 - \qi\qj - \qj^2 - \qk = -2\qk \neq 0,
\end{align*}
respectively. Hence, $\qi$ is a right zero of $P$ but $\qj$ is not. However,
$\qj$ is a left zero of~$P$.

Let $P \in \SPQ[t]$ be a split quaternion polynomial with $P\Cj{P} \neq 0$. If
we use $P$ instead of $h$ in Eq.~\eqref{eq:rotation} we obtain a one-parametric
motion in $H^2$. The trajectories of a point $[x] \in H^2$ are given by
\begin{align*}
  [x] \mapsto [P(t)x\Cj{P(t)}].
\end{align*}
They are rational curves parametrized by the parameter $t \in \R$. Since real
numbers commute with split quaternions we have $P(t)x\Cj{P(t)} = (Px\Cj{P})(t)$.
If $P(t_0)x\Cj{P(t_0)} = 0$ for some $t_0 \in \R$ then the corresponding point
$[P(t_0)x\Cj{P(t_0)}]$ is defined by continuity.

Note that we will later also consider parameter values $t \in \C$. Moreover, in
order to get algebraically closed trajectories, we should allow the parameter
value $t = \infty$ and define $P(\infty) \coloneqq p$ (the leading coefficient
of $P$).

In the following we will only consider a certain set of split quaternion
polynomials, namely those which
\begin{enumerate}
\item[1.] have a norm polynomial different from zero,
\item[2.] are monic, i.e. their leading coefficient is equal to $1 \in \R
  \subset \SPQ$ and
\item[3.] have no real or complex polynomial factors of positive degree.
\end{enumerate}

The first and second assumptions are no loss of generality from a kinematic
point of view. Polynomials with a vanishing norm polynomial do not describe
proper motions because they map all point in $H^2$ onto $\N$. Let $P \in
\SPQ[t]$ be a polynomial such that $P\Cj{P} \neq 0$ and let $p \in \SPQ$ be its
leading coefficient. If $p$ is invertible, then $P$ admits a factorization into
linear factors if and only if $p^{-1}P$ does. If $p$ is not invertible then we
can apply a proper, bijective re-parametrization to $P$ and obtain a polynomial
whose leading coefficient is invertible. The third item in our list is only a
technical assumption for now and will be dropped later on. We call split
quaternions without real or complex polynomial factors of positive degree
\emph{reduced}.

Assume that $P \in \SPQ[t]$ is a reduced, monic split quaternion polynomial
which admits a factorization into linear factors $P = (t-h_1) \cdots (t-h_n)$.
By Theorem~\ref{th:rotation}, each of the factors $t-h_i \in \SPQ[t]$ represents
a rotation with rotation center $[\Vector{t-h_i}] = [\Vector{h_i}]$. Therefore,
$P$ represents the motion obtained from the composition of these rotations --
the motion of an open kinematic chain of hyperbolic revolute joints
\cite[Chapter~8]{alkhaldi20}.

\section{Factorization Results}
\label{sec:factorization-results}

\subsection{State of the Art}
\label{sec:state-of-the-art}

In \cite{hegedus13} the authors developed an algorithm to decompose generic
``motion polynomials'' into linear ``motion polynomial factors''. Motion
polynomials are polynomials with coefficients in the non-commutative ring of
dual quaternions whose coefficients satisfy a certain quadratic constraint, the
\emph{Study condition}, that ensures that the norm polynomial is real
(c.\,f.~Section~\ref{sec:euclidean-factorization}). The factorization algorithm
can be adapted for split quaternions and will succeed in ``generic'' cases. In
Section~\ref{sec:factor-non-generic} we will develop a modified version of
this algorithm that provably succeeds whenever a factorization exists. In these
cases, it can be used to compute all factorizations.

Let $P \in \SPQ[t]$ be a reduced, monic split quaternion
polynomial and let $N \in \R[t]$ be a quadratic factor of its norm polynomial
$P\Cj{P} \in \R[t]$. Polynomial division of $P$ by $N$ yields unique polynomials
$Q$, $R \in \SPQ[t]$ such that $P = QN + R$ and $\deg(R) < \deg(N) = 2$.
Moreover, we have
\begin{align*}
  P\Cj{P} = (QN + R)\Cj{(QN + R)} = (Q\Cj{Q}N + Q\Cj{R} + R\Cj{Q})N + R\Cj{R}
\end{align*}
and $N$ divides $R\Cj{R}$. If the leading coefficient of $R = r_1t+r_0 \in
\SPQ[t]$ is invertible, then $h \coloneqq -r_1^{-1}r_0 \in \SPQ$ is the unique
right zero of $R$. Since $N$ divides $R\Cj{R} = \Cj{R}R$, $h$ is also a right
zero of $N$ and hence also of $P = QN + R$. By Lemma~\ref{lem:zero-factor}, $t-h
\in \SPQ[t]$ is a right factor of $P$ which therefore can be written as $P =
P'(t-h)$. An inductive application of the procedure above yields a (generic)
factorization algorithm \cite{li18d}.

\begin{rmk}
  A different choice of the quadratic factor $N$ yields a different right factor
  $t-h$. Therefore, the number of factorizations of $P$ depends on the number of
  distinct real quadratic factors of $P\Cj{P}$. The factorization algorithm can
  also be adapted to compute left factors: The unique left zero $g \coloneqq
  -r_0r_1^{-1}$ of $R$ yields a left factor $t-g \in \SPQ[t]$ of~$P$.
\end{rmk}

A pseudocode listing is provided in Algorithm~\ref{alg:generic-algorithm}. For
the sake of brevity, it uses the following notations: Let $G \in \SPQ$ be a
polynomial with invertible leading coefficient. Due to non-commutativity of
split quaternion multiplication we have to distinguish between \emph{left} and
\emph{right} polynomial division. The unique polynomials $Q$, $R \in \SPQ[t]$
such that $P = GQ + R$ ($P = QG + R$) are called the \emph{right (left)
  quotient} and \emph{right (left) remainder} of $P$ and $G$. They are denoted
by $Q = \rquo(P,G)$ and $R = \rrem(P,G)$ ($\lquo(P,G)$ and $\lrem(P,G)$). If $G$
is a real polynomial, then the left/right quotient and the left/right remainder
coincide, respectively. We denote them by $\quo(P,G)$ and $\rem(P,G)$. Moreover,
we use the symbol ``$\oplus$'' to denote non-commutative concatenation of
tuples, i.e. $L_1 \oplus L_2$ is the tuple which starts with the elements of the
tuple $L_1$ and ends with the elements of the tuple $L_2$.

\begin{algorithm}
  \caption{$\mathtt{GFactor}$ (Generic Factorization Algorithm)}\label{alg:generic-algorithm}
  \begin{algorithmic}[1]
    \Require $P \in \SPQ[t]$, a generic, reduced, monic, split quaternion polynomial of
    degree $n \geq 1$.
    \Ensure A tuple $L = (L_1, \ldots, L_n)$ of linear split quaternion polynomials such that $P = L_1 \cdots L_n$.
    \Statex \If{$\deg(P) = 0$}
    \State \Return ()
    \EndIf
    \State $N \leftarrow$ quadratic real factor of the norm polynomial $P\Cj{P} \in \R[t]$
    \State $R \leftarrow \rem(P,N)$
    \State $h \leftarrow$ unique right zero of $R$
    \State $P \leftarrow \lquo(P,t-h)$
    \State \Return $\mathtt{GFactor}(P) \oplus (t-h)$
  \end{algorithmic}	 	
\end{algorithm}

The term ``generic'' in Algorithm~\ref{alg:generic-algorithm} refers to $R$
having a unique right zero in each iteration. This is equivalent to the
condition $R\Cj{R} \neq 0$ because $N$ divides $R\Cj{R}$ and therefore
$\deg(R\Cj{R}) = 2$ or $R\Cj{R} = 0$. The latter one implies that the leading
coefficient of $R$ is not invertible and $h$ in
Algorithm~\ref{alg:generic-algorithm} is not well-defined. In fact, we have:

\begin{lem}[\cite{li18a}]
  \label{lem:null-lines}
  Let $R = r_1t + r_0 \in \SPQ[t]$ be a split quaternion polynomial of degree
  one with linearly independent coefficients $r_0$, $r_1 \in \SPQ$. Then the set
  $\{[R(t)] \colon t \in \R \cup \{\infty\}\}$ is the straight line spanned by
  $[r_0]$ and $[r_1]$. It is a null line if and only if $R \Cj{R} = 0$.
\end{lem}

We often will say that $R$ ``parametrizes'' the straight line spanned by $[r_0]$
and $[r_1]$. Theorem~\ref{th:affine-two-plane-of-zeros} and
Lemma~\ref{lem:null-lines} imply that a polynomial $R$ as in
Lemma~\ref{lem:null-lines} and with $R\Cj{R} = 0$ has an affine two-parametric
set of right or left zeros, respectively.

\begin{example}\label{ex:null-line-remainder-polynomial-1}
  Consider the polynomial $P_1 = t^2 - (\qi + \qj) t - \qk$. The norm polynomial
  $P_1\Cj{P_1} = t^4-1$ has the two quadratic real factors $N_1 = t^2+1$ and
  $N_2 = t^2-1$. Their respective remainder polynomials $R_1 = \rem(P_1,N_1) =
  -(\qi + \qj) t - 1 - \qk$ and $R_2 = \rem(P_1,N_2) = -(\qi + \qj) t + 1 -
  \qk$ parametrize null lines because $R_1 \Cj{R}_1 = R_2 \Cj{R}_2 = 0$.
  Moreover, $R_1$ parametrizes a right ruling of $\N$ and $R_2$ parametrizes a
  left ruling because
  \begin{align*}
    -\Cj{(\qi + \qj)} (-1 - \qk) = 0 = -(\qi + \qj) \Cj{(1 - \qk)}.
  \end{align*}
  Hence, $R_1$ has infinitely many right zeros and $R_2$ has infinitely many
  left zeros.
\end{example}

The statements on left/right zeros are no longer true if the coefficients of $R$
are linearly dependent. Recall that in this case $R$ ``parameterizes'' a single
point on~$\N$.

\subsection{Factorization of Non-Generic Polynomials}
\label{sec:factor-non-generic}

In this section we generalize the results from
Section~\ref{sec:state-of-the-art} and investigate methods to compute
factorizations into linear polynomials of reduced, monic split quaternion
polynomials, including non-generic polynomials. In addition, we provide an
algorithm for computing such factorizations. To do so, we recall a result on the
split quaternion zeros of a real polynomial.

\begin{lem}[\cite{scharler19a, huang02}]
  \label{lem:factor-real-polynomial}
  Let $P = t^2 + bt + c \in \R[t]$ be a real polynomial. The set of zeros of $P$
  in $\SPQ \setminus \R$ is given by
  \begin{align*}
    \{ \tfrac{1}{2}(-b + h_1\qi + h_2\qj + h_3\qk) \colon h_1^2-h_2^2-h_3^2 = 4c-b^2 \}.
  \end{align*}
  In particular, the set of split quaternion zeros of a quadratic real
  polynomial is two-parametric.
\end{lem}

\begin{rmk}\label{rmk:factor-real-polynomial}
  By Lemma~\ref{lem:zero-factor}, real polynomials of degree greater than one
  admit infinitely many factorizations over $\SPQ$. Depending on the sign of
  $4c-b^2$ ($<0$, $=0$ or $>0$) the set of zeros in
  Lemma~\ref{lem:factor-real-polynomial} describes a hyperboloid of one sheet, a
  quadratic cone or a hyperboloid of two sheets, respectively, in the
  three-dimensional affine space $\{ h \in \SPQ \colon \Scalar{h} = -\frac{b}{2}
  \}$ (Figure~\ref{fig:hyperboloids}). If $b = 0$, this three-space equals the
  vector space $\IM(\SPQ)$.
\end{rmk}

\begin{lem}\label{lem:null-line-remainder-polynomial}
  Consider a reduced, monic split quaternion polynomial $P \in \SPQ[t]$ and let
  $N \in \R[t]$ be a quadratic factor of its norm polynomial $P\Cj{P}$. If $R
  \coloneqq \rem(P,N)$ parametrizes a right ruling of $\N$, then there exists a
  unique split quaternion $h \in \SPQ$ such that $t-h \in \SPQ[t]$ is a right
  factor of $P$ and $N = (t - h)(t - \Cj{h})$. (The same statement holds for $R$
  parametrizing a left ruling yielding a left factor.)
\end{lem}
\begin{proof}
  Let $R = r_1t + r_0$ parametrize the right ruling $[r_0] \vee [r_1] \subset
  \N$ (c.\,f. Lemma~\ref{lem:null-lines}) and define $Q \coloneqq \quo(P,N)$.
  Because of $P = QN + R$, the sought common right zero $h$ of $P$ and $N$ must
  also be a common right zero of $N$ and $R$, and vice versa. We will show that
  there exists a unique split quaternion $h$ with this property. The proof
  proceeds with geometric descriptions of the zero sets of $N$ and $R$ in the
  affine space of split quaternions. Their unique intersection point is the
  common zero~$h$.

  A split quaternion is a zero of $R$ if and only if it solves the equation
  $-r_0=r_1x$. According to Theorem~\ref{th:affine-two-plane-of-zeros}, there
  exists an affine two-parametric plane of right zeros of $R$, let's call it $W
  \subset \SPQ$. It is parallel to the two-dimensional vector space $V
  \coloneqq \{ \lambda \Cj{r}_1 + \mu \Cj{r}_1 \qi \colon \lambda, \mu \in \R
  \}$ which parametrizes the solution set of the homogeneous equation $r_1x =
  0$. We compute
  \begin{multline*}
    (\lambda \Cj{r}_1 + \mu \Cj{r}_1 \qi) \Cj{(\lambda \Cj{r}_1 + \mu \Cj{r}_1 \qi)} = (\lambda \Cj{r}_1 + \mu \Cj{r}_1 \qi)(\lambda r_1 - \mu \qi r_1) \\=
    \lambda^2 \Cj{r}_1 r_1 - \lambda \mu \Cj{r}_1 \qi r_1 + \lambda \mu \Cj{r}_1 \qi r_1 + \mu^2 \Cj{r}_1 r_1 = (\lambda^2 + \mu^2) \Cj{r}_1r_1 = 0,
  \end{multline*}
  implying that $V$ is contained in the null cone $\NC$.

  Write $N = t^2 + bt + c$ with $b$, $c \in \R$. By applying the parameter
  transformation $t \mapsto t - \frac{b}{2}$ to $P$ we can assume without loss
  of generality that $b=0$. Set $L \coloneqq W \cap \IM(\SPQ)$ and define $Z$ as
  the set of split quaternion zeros of $N$ by
  Lemma~\ref{lem:factor-real-polynomial}.

  The affine two-plane $W$ can not be contained in $\IM(\SPQ)$ because then $V$
  is also contained in $\IM(\SPQ)$ and $\Scalar{\lambda \Cj{r}_1 + \mu \Cj{r}_1
    \qi} = 0$ for all $\lambda$, $\mu \in \R$ which, together with $r_1\Cj{r}_1
  = 0$, yields $r_1 = 0$. This is a contradiction to the assumption that $[r_0]
  \vee [r_1]$ spans a ruling of $\N$. By simply counting dimensions, we infer
  that $L$ is an affine line in the three-dimensional vector space $\IM(\SPQ)$.
  Since $W$ is parallel to $V$ and $V \subset \NC$, the line $L$ is parallel to
  a ruling of the affine cone $\NC \cap \IM(\SPQ)$. The later one is the
  asymptotic cone of $Z$ in $\IM(\SPQ)$. Therefore, the line $L$ intersects $Z$
  in precisely one point if $L$ is not contained in a tangent-plane of $\NC \cap
  \IM(\SPQ)$ (Figure~\ref{fig:hyperboloids}). We will see below that this is
  indeed the case. Hence, $L$ and $Z$ indeed intersect in a point $h \in \SPQ$
  which is the unique common right zero of $N$ and $R$ and yields the right
  factor $t-h \in \SPQ[t]$ of $P$ by Lemma~\ref{lem:zero-factor}.
  
  Thus, the proof is finished if we can show that $L$ is not contained in a
  tangent plane of the cone $\NC \cap \IM(\SPQ)$. Assume the opposite, i.e. let
  $L$ be contained in the tangent plane of an element $0 \neq p \in \NC \cap
  \IM(\SPQ)$. Then $L$ is parallel to the ruling through $p$, let's call it
  $L_p$. The \emph{Euclidean} common normal of $L$ and $L_p$ through $0 \in
  \IM(\SPQ)$ is spanned by $p \times \qi p \qi$, where $\qi p \qi$ is the
  reflection of $p$ in the plane spanned by $\qj$ and $\qk$ and ``$\times$''
  refers to the \emph{Euclidean} cross product in the three-dimensional vector
  space $\IM(\SPQ)$. Hence, there exists an $\alpha \in \R$ such that $\alpha (p
  \times \qi p \qi) \in L$. Therefore $\alpha (p \times \qi p \qi)$ is a zero of
  $R$, i.e. $r_0 = -r_1 \alpha (p \times \qi p \qi)$. Parallelity of $L$ and
  $L_p$ yields $p \in V$ and there exist $\beta$, $\gamma \in \R$ such that $p =
  \beta \Cj{r}_1 + \gamma \Cj{r}_1 \qi$. Inserting this identity into $r_0 =
  -r_1 \alpha (p \times \qi p \qi)$ and using the properties $r_1\Cj{r}_1 = 0$
  and $p \in \IM(\SPQ)$ imply linear dependency of $r_1$ and $r_0$. This is a
  contradiction to $[r_0] \vee [r_1]$ spanning a ruling of $\N$.\end{proof}

\begin{figure}
  \centering
  \begin{minipage}{0.33\linewidth}
    \centering
    \includegraphics{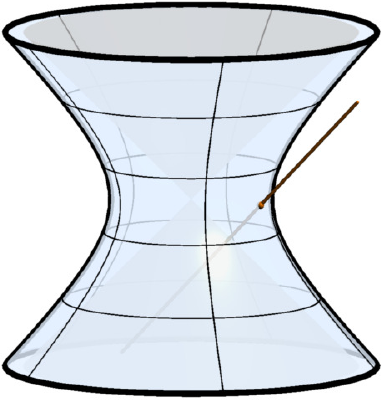}
  \end{minipage}\hfill
  \begin{minipage}{0.33\linewidth}
    \centering
    \includegraphics{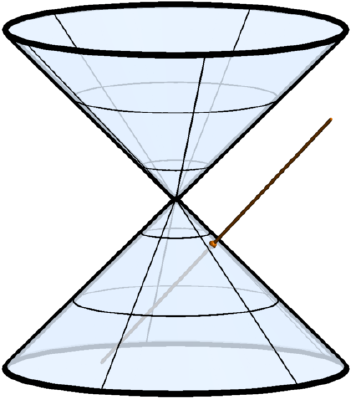}
  \end{minipage}\hfill
  \begin{minipage}{0.33\linewidth}
    \centering
    \includegraphics{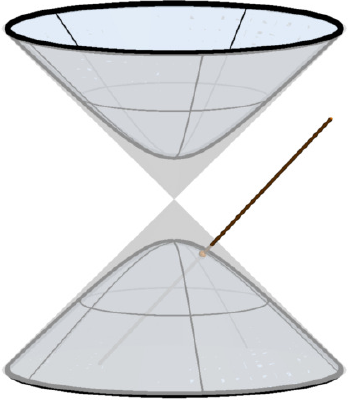}
  \end{minipage}\\[1ex]
  \begin{minipage}{0.33\linewidth}
    \centering
    \includegraphics{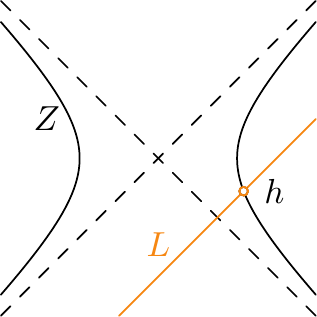}
  \end{minipage}\hfill
  \begin{minipage}{0.33\linewidth}
    \centering
    \includegraphics{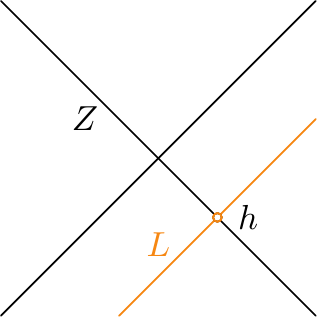}
  \end{minipage}\hfill
  \begin{minipage}{0.33\linewidth}
    \centering
    \includegraphics{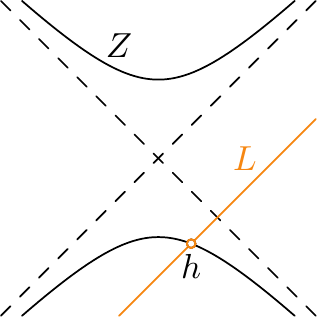}
  \end{minipage}
    \caption{Intersection of the line $L$ and the set of zeros of $N$ from the
      proof of Lemma~\ref{lem:null-line-remainder-polynomial} for $N = t^2-1$,
      $N = t^2$ and $N = t^2+1$, respectively.}
    \label{fig:hyperboloids}
\end{figure}

\begin{rmk}
  \label{rem:null-line-remainder-polynomial}
  Computing the unique zero $h$ of $P$ in
  Lemma~\ref{lem:null-line-remainder-polynomial} is easy. We have to solve the
  quadratic system resulting from $R(h) = N(h) = 0$. Plugging the solution of
  the linear sub-system $R(h) = 0$ into $N(h) = 0$ will result in a linear
  system with exactly one solution. This not only works if $R\Cj{R} = 0$ but
  also in generic cases.
\end{rmk}

\begin{example}\label{ex:null-line-remainder-polynomial-2}
  Consider the polynomial $P_1 = t^2 - (\qi + \qj) t - \qk$ from
  Example~\ref{ex:null-line-remainder-polynomial-1} and the factors $N_1 =
  t^2+1$, $N_2 = t^2-1$ of $P_1\Cj{P}_1$. Among the infinitely many right zeros of $R_1 =
  \rem(P_1,N_1) = -(\qi + \qj) t - 1 - \qk$, there is only one zero of $N_1$,
  namely $\qi \in \SPQ$. Right division of $P_1$ by $t-\qi$ yields the
  factorization $P_1 = (t-\qj) (t-\qi)$. The same result can be obtained by
  computing $\qj \in \SPQ$ as the common left zero of $R_2 = \rem(P_1,N_2)$
  and~$N_2$.
\end{example}

The only case we did not discuss yet is when the coefficients of $R$ are
linearly dependent. We can write $R = \lambda r t + \mu r = (\lambda t + \mu) r
\in \SPQ[t]$ for some $\lambda$, $\mu \in \R$ and $r \in \SPQ$. Assume that
$R\Cj{R} \neq 0$, then $r\Cj{r} \neq 0$. Since $N$ divides $R\Cj{R} = (\lambda t
+ \mu)^2 r\Cj{r}$ we infer that $N$ and $R$ have the common real zero
$-\frac{\mu}{\lambda}$. But then $P$ has a real zero which is a contradiction to
our assumption of $P$ being reduced. Hence, the case that $R$ has linearly
dependent coefficient can only appear if $R\Cj{R} = 0$. Then $R$ does not
parametrize a null line but only a point $[r] \in \N$.

\begin{example}\label{ex:constant-remainder-polynomial}
  Consider the polynomial $P_2 = t^2 + \qk$. Its norm polynomial factors as
  $P_2\Cj{P_2} = N_1 N_2 = (t^2+1)(t^2-1)$. The respective remainder polynomials
  $R_1 = \rem(P_2,N_1) = -1+\qk$ and $R_2 = \rem(P_2,N_2) = 1+\qk$ do not have a
  zero in $\SPQ$. A straightforward computation shows that $P_2$ does neither
  have a right zero nor a left zero and thus does not admit a factorization into
  linear factors.
\end{example}

The next lemma shows that only remainder polynomials of a certain type can occur in
the presence of linear right factors.

\begin{lem}\label{lem:right-factor-no-left-ruling}
  Consider a reduced, monic split quaternion polynomial $P \in \SPQ[t]$. Let
  $t-h \in \SPQ[t]$ be a left/right factor of $P$ and let $N = (t-h)(t-\Cj{h}) \in
  \R[t]$. Then the remainder polynomial $R = \rem(P,N)$ parametrizes either a
  non-null line or a left/right ruling of $\N$.
\end{lem}
\begin{proof}
  We only prove the statement for right factors. If $R\Cj{R} \neq 0$, the
  coefficients of $R$ are linearly independent by the considerations following
  Example~\ref{ex:null-line-remainder-polynomial-2}. Moreover, $R$ parametrizes
  a non-null line by Lemma~\ref{lem:null-lines}.

  Next, we assume $R\Cj{R} = 0$. The split quaternion $h \in \SPQ$ is a right
  zero of $N$ as well as of $P$ by Lemma~\ref{lem:zero-factor}. Define $Q =
  \quo(P,N)$, then $h$ is a right zero of $R = P - QN = r_1 t + r_0$, i.\,e.
  $-r_0 = r_1 h$. Provided $r_0$ and $r_1$ are linearly independent, the points
  $[r_0]$ and $[r_1]$ are well defined and, by
  Corollary~\ref{cor:left-right-ruling}, span a right ruling. We proceed by
  showing linear independence of $r_0$ and $r_1$. Assume, on the contrary, that
  $R = \lambda r t + \mu r \in \SPQ[t]$ for some $\lambda$, $\mu \in \R$ and $r
  \in \SPQ \setminus \{0\}$. We have $0 = R(h) = (\lambda r h + \mu r) = r
  (\lambda h + \mu)$ and $\lambda h + \mu$ is a zero divisor. If $\lambda = 0$
  this implies that $\mu = 0$ and moreover $R = 0$. Consequently, $P = QN$ is
  not reduced, which is a contradiction to our assumption. Hence, $\lambda \neq
  0$ and $h + \frac{\mu}{\lambda}$ is a zero divisor, i.\,e. $0 = (h +
  \frac{\mu}{\lambda}) \Cj{(h + \frac{\mu}{\lambda})} = (h +
  \frac{\mu}{\lambda}) (\Cj{h} + \frac{\mu}{\lambda}) =
  N(-\frac{\mu}{\lambda})$. Again, we obtain a contradiction since we infer that
  $P = QN + R$ has the real zero $-\frac{\mu}{\lambda}$ and is not reduced.
\end{proof}

Assume that $P$ admits a factorization into linear factors $P = (t-h_1) \cdots
(t-h_n)$. The rightmost factor $t-h_n$ can be computed by using the real
polynomial $N = (t-h_n)(t-\Cj{h}_n)$ and $R = \rem(P,N)$. By
Lemma~\ref{lem:right-factor-no-left-ruling}, $R$ parametrizes either a non-null
line and we run one iteration of Algorithm~\ref{alg:generic-algorithm}, or $R$
parametrizes a right ruling of $\N$ and we compute $t-h_n$ as explained in
Remark~\ref{rem:null-line-remainder-polynomial}. The complete factorization can
be obtained by an iterative procedure using the real polynomials
$(t-h_{n-1})(t-\Cj{h}_{n-1}),\ldots,(t-h_1)(t-\Cj{h}_1) \in \R[t]$.

\begin{algorithm}
  \caption{$\mathtt{NGFactor}$ (Non-Generic Factorization Algorithm)}\label{alg:non-generic-algorithm}
  \begin{algorithmic}[1]
    \Require $(P,F)$, where $P \in \SPQ[t]$ is a reduced, monic, split quaternion polynomial of
    degree $n \geq 1$ and $F = (F_1,\ldots,F_n)$ is a tuple of quadratic real
    polynomials such that $P\Cj{P} = F_1 \cdots F_n$.
    \Ensure A tuple $L = (L_1, \ldots, L_n)$ of linear split quaternion
    polynomials such that $P = L_1 \cdots L_n$ and $L_i\Cj{L}_i = F_i$ for $i = 1,\ldots,n$.
    \Statex
    \If{$\deg(P) = 0$}
    \State \Return ()
    \EndIf
    \State $N \leftarrow F_n$, $F \leftarrow (F_1,\ldots,F_{n-1})$
    \State $R = r_1t + t_0 \leftarrow \rem(P,N)$
    \If{$R\Cj{R} \neq 0$}
    \State $h \leftarrow$ unique right zero of $R$
    \State $P \leftarrow \lquo(P,t-h)$
    \State \Return $\mathtt{NGFactor}(P,F) \oplus (t-h)$
    \ElsIf{$r_1 \Cj{r}_0 \neq 0$} \Comment $[r_0] \vee [r_1]$ does not lie on a left ruling
    \State $h \leftarrow$ unique common right zero of $N$ and $R$
    \Comment{compute via Remark~\ref{rem:null-line-remainder-polynomial}}
    \State $P \leftarrow \lquo(P,t-h)$
    \State \Return $\mathtt{NGFactor}(P,F) \oplus (t-h)$
    \Else
    \State \Return No factorization with respect to $F$ exists.
    \EndIf
  \end{algorithmic}	 	
\end{algorithm}

These considerations are transformed into
Algorithm~\ref{alg:non-generic-algorithm} for computing factorizations also in
non-generic cases. It attempts to recursively compute a factorization from a
given sequence $(F_1,\ldots,F_n)$ of quadratic real factors of the norm
polynomial $P\Cj{P}$. In doing so, it only tries to find \emph{right} factors.
If this fails, it stops and returns a factorization of a right factor. The next
theorem justifies the restriction to right factors only.

\begin{thm}\label{th:alg-finds-factorization}
  A reduced, monic split quaternion polynomial $P \in \SPQ[t]$ of degree $n \ge
  1$ admits a factorization into linear factors if an only if
  Algorithm~\ref{alg:non-generic-algorithm} finds a factorization for some tuple
  $F = (F_1,\ldots,F_n)$ consisting of quadratic real polynomials such that
  $P\Cj{P} = F_1 \cdots F_n$.
\end{thm}
\begin{proof}
  We have already argued the correctness of
  Algorithm~\ref{alg:non-generic-algorithm} so that a factorization exists, if
  the algorithm does not stop prematurely. On the other hand, the considerations
  after the proof of Lemma~\ref{lem:right-factor-no-left-ruling} show that
  Algorithm~\ref{alg:non-generic-algorithm} will recursively find the linear
  factors of the factorizations $P = (t-h_1) \cdots (t-h_n)$ for the input $F =
  (F_1,\ldots,F_n)$ with $F_i = (t-h_i)(t-\Cj{h}_i)$, $i \in \{1,\ldots,n\}$.
\end{proof}

\begin{rmk}
  Algorithm~\ref{alg:non-generic-algorithm} relies on the possibility to compute
  factorizations of the real norm polynomial $P\Cj{P}$. In general, this
  requires numerical approximation which is not naturally incorporated into our
  otherwise symbolic algorithm. Nonetheless, it should be considered. We believe
  that this should be possible in a satisfactory way but refrain from pursuing
  this question in this article.
\end{rmk}

\begin{example}\label{ex:null-line-remainder-polynomial-3}
  Once more, we consider the polynomial $P_1 = t^2 - (\qi + \qj) t - \qk$ from
  Example~\ref{ex:null-line-remainder-polynomial-1}. As we have seen in
  Example~\ref{ex:null-line-remainder-polynomial-2},
  Algorithm~\ref{alg:non-generic-algorithm} computes a factorization for the
  input $(P,(N_2,N_1))$. It will fail for the input $(P,(N_1,N_2))$ because $R_2
  = -(\qi + \qj) t - 1 - \qk$ parametrizes a left ruling of $\N$ and thus has no
  right zeros. Of course, one could device a left version of
  Algorithm~\ref{alg:non-generic-algorithm} as well.
\end{example}

\begin{example}
  For every integer $n > 1$, there exist polynomials of degree $n$ without
  linear factors. One example is $t^n + \qi + \qk$. Its norm polynomial equals
  $t^{2n}$. Thus, the only quadratic factor to use in
  Algorithm~\ref{alg:non-generic-algorithm} is $t^2$. We have $\rem(t^n + \qi +
  \qk,t^2) = \qi + \qk$ -- a constant polynomial which does not have any zero.
\end{example}

\subsection{Geometry of the Factorization Algorithm}
\label{sec:geometry-of-algorithm}

In this section we will investigate the ``geometry'' of
Algorithm~\ref{alg:non-generic-algorithm} in order to clarify the cause of
non-factorizability. By Theorem~\ref{th:alg-finds-factorization}, a reduced,
monic split quaternion polynomial $P$ does not admit a factorization into linear
factors if and only if the algorithm fails for any order of the quadratic real
factors $F_i$. Let $P = P' (t-h_m)\cdots(t-h_n)$ for some $m \in \{1,\ldots,n\}$
and a polynomial $P' \in \SPQ[t]$ that does not have a linear right factor.
Moreover, let $N \in \R[t]$ be a quadratic factor of $P'\Cj{P'}$. According to
Lemma~\ref{lem:null-line-remainder-polynomial} and the considerations in
Section~\ref{sec:state-of-the-art}, $R = \rem(P',N)$ can not have a right zero
and must either parametrize a left ruling of $\N$ or a point $[r] \in \N$, i.e.
$[R(t)] = [r] \in \N$ for all $t \in \R \cup \{ \infty \}$. Let $t_1$, $t_2 \in
\C$ be such that $N = (t-t_1)(t-t_2)$. For the time being we assume that $t_1$,
$t_2 \in \R$.
For $i = 1,2$ we have
\begin{align*}
  P(t_i) = P'(t_i) (t_i-h_m)\cdots(t_i-h_n)
\end{align*}
and the point $[P(t_i)] \in \N$ is the image of $[P'(t_i)] \in \N$ under the
Clifford right translation $[x] \mapsto [x (t_i-h_m)\cdots(t_i-h_n)]$, c.\,f.
Corollary~\ref{cor:left-right-ruling} and Remark~\ref{rmk:clifford-translation}.
The point $[P'(t_i)]$ is contained in the pre-image of $[P(t_i)]$ if the
Clifford right translation is singular and equals $[P(t_i)]$ otherwise.
\emph{The failure of Algorithm~\ref{alg:non-generic-algorithm} is equivalent to
  the line $[P(t_1)] \vee [P(t_2)]$ being contained in the pre-image of
  $[P'(t_1)] \vee [P'(t_2)]$, which is either a left ruling of $\N$ or
  degenerates to a single point $[r] \in \N$.}

\begin{example}\label{ex:algorithm-stops-1}
  Consider the polynomial $P_3 = t^3 - \qi t^2 + \qk t - \qj$. The norm
  polynomial is $P_3 \Cj{P_3} = (t^2+1)^2 (t^2-1)$ and we can run
  Algorithm~\ref{alg:non-generic-algorithm} with the following triples:
  \begin{align*}
    (t^2-1,t^2+1,t^2+1), \ (t^2+1,t^2-1,t^2+1), \ (t^2+1,t^2+1,t^2-1).
  \end{align*}
  The latter one yields the remainder polynomial $R_1 = \rem(P_3,t^2-1) =
  (1+\qk) t - \qi - \qj$ and the right factor $t-\qj$ and $P_3' =
  \lquo(P_3,t-\qj) = t^2 + (\qj - \qi) t + 1$ in the first iteration. In the
  second iteration we obtain the remainder polynomial $\rem(P_3',t^2+1) = (\qj -
  \qi) t$ which parametrizes only the point $[\qj - \qi] \in \N$ and the
  algorithm stops. The polynomial $R_1$ parametrizes a right ruling of $\N$
  which is contained in the pre-image of $[\qj - \qi]$.

  Both of the other two triples yield the remainder polynomial $R_2 =
  \rem(P_3,t^2+1) = (\qk-1) t + \qi - \qj$ and the right factor $t-\qi$ and
  $P_3' = \lquo(P_3,t-\qi) = t^2+\qk$ in the first iteration. We have already
  seen in Example~\ref{ex:constant-remainder-polynomial} that $t^2+\qk$ does not
  admit a factorization because both remainder polynomials $\rem(P_3',t^2-1) = 1
  + \qk$ and $\rem(P_3',t^2+1) = -1 + \qk$ of the second iteration parametrize
  only a point on $\N$. The right ruling of $\N$ parametrized by $R_2$ is
  contained in the pre-image of $[-1+\qk]$ and $[1+\qk]$, respectively, with
  respect to the according Clifford right translations.

  Hence, Algorithm~\ref{alg:non-generic-algorithm} fails for all triples of
  quadratic real factors and $P_3$ does not admit a factorization.
\end{example}

If $t_1 = \overline{t}_2 \in \C \setminus \R$, then the coefficients of $R(t_1)$
and $R(t_2)$ might be complex numbers. However, all definitions and results from
Section~\ref{sec:split-quaternions-and-geometry} can be translated to the ring
of split quaternions with coefficients in $\C$. This has been done in
\cite{siegele20} with the Hamiltonian quaternions. Considering complex numbers
as coefficients they are isomorphic to the (complex) split quaternions. In
addition, we extend the definition of the null cone $\NC$ as well as the null
quadric $\N$ by also considering the split quaternions with complex coefficients
fulfilling their defining equations, respectively. Note that the affine plane in
Theorem~\ref{th:affine-two-plane-of-zeros} extends to an affine plane of
\emph{complex} dimension two in this setup.

If a split quaternion polynomial admits a factorization into linear factors then
it must have a linear left/right factor and therefore a left/right zero by
Lemma~\ref{lem:zero-factor}. Both potential remainder polynomials of $P_3$ in
Example~\ref{ex:algorithm-stops-1} $\rem(P_3,t^2-1) = (1+\qk) t - \qi - \qj$ and
$\rem(P_3,t^2+1) = (\qk-1) t + \qi - \qj$ parametrize right rulings of $\N$.
Hence, $P_3$ has no left zeros and we could state that is does not admit a
factorization without the necessity to run
Algorithm~\ref{alg:non-generic-algorithm} with all triples of quadratic real
factors of $P_3 \Cj{P_3}$. This result can be easily generalized.

\begin{thm}\label{th:no-factorization-exists}
  Consider a reduced, monic split quaternion polynomial $P \in \SPQ[t]$. If all
  sets $\{ [R(t)] \subset \P(\SPQ) \colon t \in \R \cup \{ \infty \} \}$ parametrized by its
  remainder polynomials after division by quadratic real factors of $P\Cj{P}$
  belong to only on family of rulings of $\N$, then $P$ does not admit a
  factorization into linear factors.
\end{thm}
\begin{proof}
  $P$ admits a factorization into linear factors if and only if $\Cj{P}$ does.
  Moreover, conjugation of split quaternions interchanges the two families of
  rulings of $\N$. Hence, we can, without loss of generality, assume that all
  remainder polynomials of $P$ parametrize (subsets of) left rulings of $\N$.
  Then $P$ has no linear right factor by
  Lemma~\ref{lem:right-factor-no-left-ruling} and therefore does not admit a
  factorization into linear factors.
\end{proof}

\subsection{Factorizing Hyperbolic Motions}
\label{sec:factorizing-motions}

Since there exist split quaternions which do not admit a factorization into
linear factors we can not unconditionally use the standard factorization theory
in order to factor motions described by such polynomials. However, the
projective nature of our kinematic model allows to multiply a split quaternion
polynomial $P$ by a real polynomial $T \in \R[t]$. The product $TP$ describes
the same motion as $P$ and may admit a factorization into linear factors. In
this case, we can still decompose the initial motion into rotations even though
$P$ does not admit a factorization. In the following we will show that such a
polynomial $T$ always exists.

\begin{thm}\label{th:multiplication-trick}
  Consider a reduced, monic split quaternion polynomial $P \in \SPQ[t]$. There
  exists a real polynomial $T \in \R[t]$ such that $TP$ admits a factorization
  into linear factors.
\end{thm}
\begin{proof}
  If $P$ already admits a factorization we can choose $T = 1$. Assume that $P$
  does not admit a factorization into linear factors. Without loss of generality
  we can assume that $P$ has no right factors. Otherwise we run
  Algorithm~\ref{alg:non-generic-algorithm} with a suitable tuple of quadratic
  real factors of $P\Cj{P}$ and continue with the input polynomial in the
  iteration where the algorithm stops. Let $N = (t-t_1)(t-t_2) \in \R[t]$ be a
  quadratic, real factor of $P\Cj{P}$ with $t_1$, $t_2 \in \C$. By the
  considerations at the beginning of Section~\ref{sec:geometry-of-algorithm},
  $[P(t_1)] \vee [P(t_2)]$ is either a left ruling of $\N$ or $[P(t_1)] =
  [P(t_2)] \in \N$.

  Consider $H = t - h \in \SPQ[t]$ and denote by $z$, $\overline{z} \in \C
  \setminus \R$ the conjugate complex zeros of $H\Cj{H} = (t-z)(t-\overline{z})
  \in \R[t]$. We will later argue that we can select $H$ such that
  \begin{enumerate}
    \item[1.] \label{it1} $P\Cj{P}(z) \neq 0 \neq P\Cj{P}(\overline{z})$,
    \item[2.]\label{it2} $[H(t_1)P(t_1)] \neq [H(t_2)P(t_2)]$,
    \item[3.]\label{it3} $[H(z)P(z)] \neq [H(\overline{z})P(\overline{z})]$ and
      $[H(z)P(z)] \vee [H(\overline{z})P(\overline{z})]$ is no left ruling of~$\N$.
  \end{enumerate}

\begin{figure}
  \centering
  \begin{overpic}[width=0.3\textwidth]{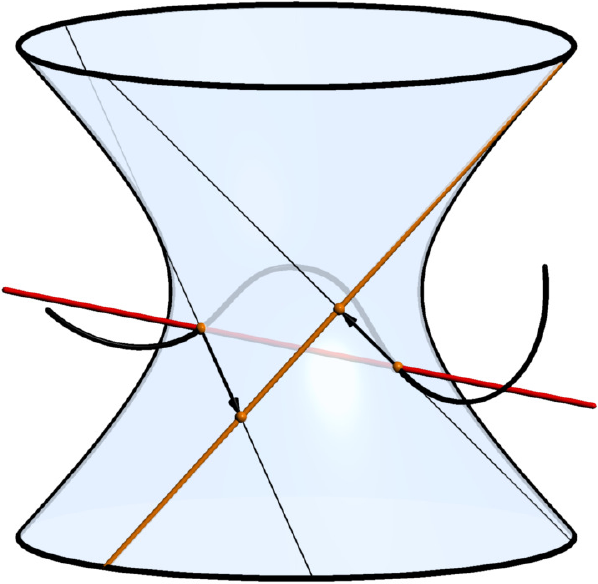}
    \put(90,22){\textcolor{red}{$L$}}
    \put(25,5){\textcolor{orange}{$K$}}
    \put(85,55){$HP$}
  \end{overpic}
    \caption{Illustration for the choice of $H$: The non-null line $L \coloneqq [H(z)P(z)]
      \vee [H(\overline{z})P(\overline{z})]$ is mapped to the right ruling
      $K$ when dividing off a the left factor $t-h_l$. The points
      $[H(z)P(z)]$ and $[H(\overline{z})P(\overline{z})]$ are mapped by the
      respective Clifford left translations.}
    \label{fig:gfigure}
\end{figure}

  By Remark~\ref{rmk:clifford-translation} and Item~2, the points
  $[H(t_1)P(t_1)]$ and $[H(t_2)P(t_2)]$ span the left ruling through $[P_1]$ and
  $[P_2]$. Therefore, $\rem(HP,N)$ yields a left factor $t - h_l \in \SPQ$ of
  $HP = (t-h_l) P'$ according to
  Lemma~\ref{lem:null-line-remainder-polynomial}.

  Item~1 implies that the Clifford left translations $\eta\colon [x] \mapsto
  [(z-h_l)x]$ and $\overline{\eta}\colon [x] \mapsto [\overline{z}-h_l]$ are
  non-singular. By the considerations at the beginning of
  Section~\ref{sec:geometry-of-algorithm}, $[P'(z)]$ and $[P'(\overline{z})]$
  are the images of $[H(z)P(z)]$ and $[H(\overline{z})P(\overline{z})]$ under
  the inverse Clifford left translations $\eta^{-1}$ and
  $\overline{\eta}^{-1}$, respectively.

  From Item~3 we infer that the left rulings through $[H(z)P(z)]$ and
  $[H(\overline{z})P(\overline{z})]$ are different and $[P'(z)] \vee
  [P'(\overline{z})]$ is either a right ruling of $\N$ or a non-null line.
  Either way, $\rem(P',H\Cj{H})$ yields a right factor $t - h_r \in \SPQ$ of $P'
  = P'' (t-h_r)$. We define $T_1 \coloneqq H\Cj{H} \in \R[t]$ and obtain
  \begin{align*}
    T_1P = (t-\Cj{h})(t-h)P = (t-\Cj{h})(t-h_l)P''(t-h_r)
  \end{align*}
  where $\deg(P'') = \deg(P)-1$. We continue by running
  Algorithm~\ref{alg:non-generic-algorithm} with $P''$ and repeat the procedure
  above whenever the algorithm stops. Finally, we define $T \coloneqq T_1 \cdots
  T_m$ as the product of appearing real polynomials.

  It remains to justify the possibility to choose $H$ such that the conditions
  in Items~1--3 hold.

  The a priori degree of freedom for the choice of $H$ is four, namely the four
  real coefficients of $h$. We will show that each of the conditions in
  Items~1--3 accounts for the loss of at most two or three degrees of freedom.

  Item~1: $H\Cj{H} \in \R[t]$ is an irreducible, quadratic polynomial. The
  number of irreducible, quadratic factors of $P\Cj{P}$ is finite. Two different
  irreducible, quadratic real polynomials do not have any common split
  quaternion zeros (c.\,f. Lemma~\ref{lem:factor-real-polynomial}). Hence, our
  choice of $H$ is restricted by avoiding finitly many zero-sets of quadratic
  split-quaternions polynomials which are two-parametric, again by
  Lemma~\ref{lem:factor-real-polynomial}.

  Item~2: Consider the map
  \begin{align*}
    \varphi \colon \SPQ \setminus \{ g \in \SPQ \colon g P(t_1) = 0 \} \to \N \colon x \mapsto [x P(t_1)].
  \end{align*}
  Its image is the left ruling of $\N$ through $[P(t_1)]$ which also contains
  $[P(t_2)]$. By Theorem~\ref{th:affine-two-plane-of-zeros}, the pre-image of
  $[P(t_2)]$ is a one-parametric union of two-dimensional affine planes, each of
  them corresponding to a representative of $[P(t_2)]$. This is even true for
  $t_2 \in \C \setminus \R$ as the subspace of real points in an affine plane of
  complex dimension two is at most of real dimension two. We choose $H$ such
  that the split quaternion $\Cj{H(t_2)} H(t_1) \in \SPQ$ is not contained in
  the union of these planes. This is possible because $\Cj{H(t_2)} H(t_1)$ has
  real split quaternion coefficients
  \begin{align*}
    \Cj{H(t_2)} H(t_1) = (t_2 - \Cj{h})(t_1 - h) = \underbrace{t_1t_2 + h\Cj{h}}_{\in \R} - \underbrace{(t_1+t_2)}_{\in \R} h
  \end{align*}
  and, for varying coefficients of $h$, parametrizes a real four-parametric set
  in the vector space of real split quaternions. By the choice of $H$, we have
  $[P(t_2)] \neq [\Cj{H(t_2)} H(t_1) P(t_1)]$ and therefore
  \begin{align*}
    [H(t_2) P(t_2)] \neq [H(t_2)\Cj{H(t_2)} H(t_1) P(t_1)] = [H(t_1) P(t_1)].
  \end{align*}

  Item~3: By Item~1, we have $[P(z)]$, $[P(\overline{z})] \notin \N$. The span
  of $[H(z)]$ and $[H(\overline{z})]$ is a straight line $[H(z)] \vee
  [H(\overline{z})]$ which is no ruling of $\N$ because $[H(\infty)] = [1]
  \notin \N$. By Remark~\ref{rmk:clifford-translation}, the point $[H(z)P(z)]$
  is contained in the right ruling through $[H(z)]$ and
  $[H(\overline{z})P(\overline{z})]$ is contained in the right ruling through
  $[H(\overline{z})]$. These two rulings are different and hence $[H(z)P(z)]
  \neq [H(\overline{z})P(\overline{z})]$. It only remains to ensure that
  $[H(z)P(z)] \vee [H(\overline{z})P(\overline{z})]$ is no left ruling of $\N$,
  i.\,e.
  \begin{align*}
    H(z)P(z)\Cj{(H(\overline{z})P(\overline{z}))} =
    H(z)P(z)\Cj{P(\overline{z})}\Cj{H(\overline{z})}\neq 0.
  \end{align*}
  A direct computation shows that $P(z)\Cj{P(\overline{z})}$ and
  $H(z)P(z)\Cj{P(\overline{z})}\Cj{H(\overline{z})}$ are of the form
  \begin{align*}
    P(z)\Cj{P(\overline{z})} &= P(z)\overline{\Cj{P(z)}} = p_0 + \underbrace{(p_1 \qi + p_2 \qj + p_3 \qk)}_{\eqqcolon p} \ci \neq 0, \\
    H(z)P(z)\Cj{P(\overline{z})}\Cj{H(\overline{z})} &= H(z)P(z)\overline{\Cj{(H(z)P(z))}} = q_0 + (q_1 \qi + q_2 \qj + q_3 \qk) \ci
  \end{align*}
  where $p_\ell$, $q_\ell \in \R$ for $\ell = 0,\ldots,3$ and $\ci \in \C$ is
  the imaginary unit. If $p = 0$, then
  \begin{align*}
    H(z)P(z)\Cj{P(\overline{z})}\Cj{H(\overline{z})} = H(z)p_0\Cj{H(\overline{z})} =
    p_0H(z)\Cj{H(\overline{z})} \neq 0
  \end{align*}
  since $[H(z)] \vee [H(\overline{z})]$ is no ruling of $\N$. Else, another
  straightforward computation, which uses the condition $H(z)\Cj{H(z)} = 0$ for
  simplification, shows that
  \begin{align*}
    q_0 = 2 z_1 \langle h, p \rangle + 2 p_0 + z_1^2,
  \end{align*}
  where $z_1 \in \R$ is the imaginary part of the complex number $z \in \C$. If
  $q_0 = 0$, we replace $h$ by another element in the two-parametric set of
  zeros of $H\Cj{H}$. This does not change $z$, $z_1$ and $p$. The equation
  $H\Cj{H} = 0$ defines a quadratic surface in the three-dimensional affine
  space (c.\,f. Remark~\ref{rmk:factor-real-polynomial}), given by an equation
  of the shape $ah_0 + b = 0$ with $a$, $b \in \R$. Because of $p \neq 0$, it is
  different from the affine space $\{ x \in \SPQ \colon 2 z_1 \langle x, p
  \rangle + 2 p_0 + z_1^2 = 0\}$. Hence, the latter three-space intersects the
  quadric $H\Cj{H} = 0$ (at most) in a conic section. This shows that there
  exists a replacement for $h$ such that $q_0 \neq 0$ and thus
  $H(z)P(z)\Cj{P(\overline{z})}\Cj{H(\overline{z})} \neq 0$. In addition, we can
  guarantee that the conditions of Items~1 and 2 still hold if we choose our new
  $h$ near the initial one due to continuity of the polynomial inequalities in
  Item~1 and~2.
\end{proof}

\begin{example}\label{ex:algorithm-stops-2}
  Consider the polynomial $P_3 = t^3 - \qi t^2 + \qk t - \qj$ from
  Example~\ref{ex:algorithm-stops-1}. It does not admit a factorization.
  Algorithm~\ref{alg:non-generic-algorithm} stops after one iteration for the
  triple $(t^2+1,t^2-1,t^2+1)$ and yields only one linear right factor $P_3 =
  (t^2+\qk)(t-\qi)$. We define $H = t-2\qi \in \SPQ$, $T = H\Cj{H} = t^2+4 \in
  \R[t]$ and obtain the factorization
  \begin{multline*}
    T P_3 = (t+2\qi) (t-2\qi) (t^2+\qk) (t-\qi) \\=
    (t+2\qi) (t-\tfrac{3}{4}\qi+\tfrac{5}{4}\qj) (t+\tfrac{61}{60}\qi-\tfrac{11}{60}\qj) (t-\tfrac{34}{15}\qi-\tfrac{16}{15}\qj) (t-\qi).
  \end{multline*}
\end{example}

From a geometric point of view, a split quaternion polynomial $P \in \SPQ[t]$
does neither have a right factor nor a left factor precisely if each quadratic,
real factor $N = (t-t_1)(t-t_2)$ of $P\Cj{P}$ yields a remainder polynomial $R =
\rem(P,N)$ that parametrizes only a point of $\N$, i.e. $[P(t_1)] = [P(t_2)] \in
\N$. Multiplication by the polynomials $H = t-h \in \SPQ$ in the proof of
Theorem~\ref{th:multiplication-trick} separates these ``two'' points and in
addition guarantees that the line $[H(z)P(z)]
\vee[H(\overline{z})P(\overline{z})]$ does not degenerate to only a point when
dividing off the left factor of $HP$ obtained by $N$. Hence, at least one
further iteration of Algorithm~\ref{alg:non-generic-algorithm} is successful
for~$HP$.

Pseudo-code for factorization via multiplying with a real polynomial is provided
in Algorithm~\ref{alg:all-factor}. Lines~1 to \ref{li:xxx} are essentially
identical to Algorithm~\ref{alg:non-generic-algorithm}, the novel part starts
with Line~\ref{li:yyy}.

\begin{algorithm}
  \caption{$\mathtt{AllFactor}$ (Factorization via Multiplication with Real
    Polynomial)}
  \label{alg:all-factor}
  \begin{algorithmic}[1]
    \Require $(P,F,T)$, where $P \in \SPQ[t]$ is a reduced, monic, split quaternion polynomial of
    degree $n \geq 1$, $T \in \R[t]$ is a real polynomial (initially $T=1$) and
    $F = (F_1,\ldots,F_n)$ is a tuple of quadratic real polynomials such that $P\Cj{P} = F_1 \cdots F_n$.
    \Ensure A pair $(L,T)$, where $L = (L_1, \ldots, L_m)$ is a tuple of linear split quaternion
    polynomials and $T$ is a real polynomial such that $TP = L_1 \cdots L_m$ with $m
    \geq n$.
    \Statex
    \If{$\deg(P) = 0$}
    \State \Return ()
    \EndIf
    \State $N \leftarrow F_n$, $F \leftarrow (F_1,\ldots,F_{n-1})$
    \State $R = r_1t + t_0 \leftarrow \rem(P,N)$
    \If{$R\Cj{R} \neq 0$}
    \State $h \leftarrow$ unique right zero of $R$
    \State $P \leftarrow \lquo(P,t-h)$
    \State \Return $(\mathtt{ALLFactor}(P,F,T) \oplus (t-h),T)$
    \ElsIf{$r_1 \Cj{r}_0 \neq 0$} \Comment $[r_0] \vee [r_1]$ does not lie on a left ruling
    \State $h \leftarrow$ unique common right zero of $N$ and $R$
    \Comment{compute via Remark~\ref{rem:null-line-remainder-polynomial}}
    \State $P \leftarrow \lquo(P,t-h)$
    \State \Return $(\mathtt{AllFactor}(P,F,T) \oplus (t-h),T)$\label{li:xxx}
    \Else \label{li:yyy}
    \State Choose $H = t-h$ according to Item~1--3 in the proof of Theorem~\ref{th:multiplication-trick}.
    \State $R_l \leftarrow \rem(HP,N)$
    \State $h_l \leftarrow$ unique common left zero of $N$ and $R_l$
    \Comment{compute via Remark~\ref{rem:null-line-remainder-polynomial}}
    \State $P' \leftarrow \rquo(HP,t-h_l)$
    \State $R_r \leftarrow \rem(P',H\Cj{H})$
    \State $h_r \leftarrow$ unique common right zero of $H\Cj{H}$ and $R_r$
    \Comment{compute via Remark~\ref{rem:null-line-remainder-polynomial}}
    \State $P'' \leftarrow \lquo(P',t-h_r)$
    \State $T \leftarrow T H\Cj{H}$
    \State \Return $((t-\Cj{h},t-h_l) \oplus \mathtt{AllFactor}(P'',F,T) \oplus (t-h_r),T)$
    \EndIf
  \end{algorithmic}	 	
\end{algorithm}

\subsection{Factorization of Euclidean Motions}
\label{sec:euclidean-factorization}

The group of Euclidean motions in three-dimensional Euclidean space can be
parametrized by a certain subset of (Hamiltonian) dual quaternions. Hamiltonian
quaternions $\H$ are defined in the same manner as split quaternions with
only some changes of sign in the generating relations:
\begin{align*}
  \qi^2 = \qj^2 = \qk^2 = \qi \qj \qk = -1.
\end{align*}
Unlike the split quaternions, they define a skew field. The non-commutative ring
of dual quaternions is defined as $\DH \coloneqq \H[\eps] / \langle \eps^2
\rangle$, the quotient of the polynomial ring $\H[\eps]$ and the ideal generated
by $\eps^2$ where the indeterminate $\eps$ with the property $\eps^2 = 0$
commutes with the complex units $\qi$, $\qj$ and $\qk$. The norm of a dual
quaternion $p + \eps d \in \DH$ with $p$, $d \in \H$ equals $(p + \eps d)(\Cj{p}
+ \eps \Cj{d}) = p\Cj{p} + \eps (p\Cj{d} + d\Cj{p})$. It is an element of
$\R/\langle \eps^2 \rangle$ -- the ring of dual numbers. The group of Euclidean
motions $\SE[3]$ is isomorphic to the quadric defined by the equation $p\Cj{d} +
d\Cj{p} = 0$ minus the subspace $\{ p + \eps d \in \DH \colon p = 0 \}$ in the
seven-dimensional projective space of dual quaternions $\P(\DH)$. A rational
motion in the Euclidean three-space can be represented by a dual quaternion
polynomial $P + \eps D \in \DH[t]$ with $P$, $D \in \H[t]$ that fulfills the
Study condition $P\Cj{D} + D\Cj{P} = 0$ as well as $P \neq 0$. We call such
polynomials \emph{motion polynomials}. Their norm polynomials equal $(P + \eps
D)(\Cj{P} + \eps \Cj{D}) = P\Cj{P}$ and are real. For more details we refer
to~\cite{hegedus13}.

Some of our ideas and results can be applied for the factorization of motion
polynomials as well and shed new light on the already quite sophisticated
Euclidean theory \cite{li15b}. Consider a reduced motion polynomial $P + \eps D
\in \DH[t]$ and its norm polynomial $(P + \eps D)(\Cj{P} + \eps \Cj{D}) =
P\Cj{P} + \eps (P\Cj{D} + D\Cj{P}) = P\Cj{P}$. Let $N$ be a quadratic factor of
$P\Cj{P}$. Generically, $R = \rem(P + \eps D,N)$ has a unique left/right zero
which yields a linear left/right factor of $P + \eps D$. If $P \in \H[t]$ has a
quadratic, real, irreducible factor $M \in \R[t]$ we can choose $N = M$. Then
$R$ is of the form $R = \eps R'$ for some $R' \in \H[t]$ and has no unique zero
anymore. This issue has already been addressed in \cite{li15b} and could be
overcome by computing a common quaternion left/right zero of $N$ and $D$ which
again yields a linear left/right factor of $P + \eps D$. In the following we
will consider the situation above from a more geometric point of view and
compute a common dual quaternion zero of $N$ and $R$. Similar to
Lemma~\ref{lem:factor-real-polynomial}, the set of quaternion zeros of the
irreducible, real polynomial $N$ is a Euclidean sphere in some three-dimensional
affine space \cite{huang02}.

Let $p \in \H$ be a zero of $N$. A straightforward computation shows that $p +
\eps d \in \DH$ is a zero of $N$ precisely if $d \in \H$ fulfills two linear
equations induced by the coefficients of $N$. Hence, the set of dual quaternion
zeros of $N$ can be viewed as a four-dimensional ``cylinder'' $C$ in the
eight-dimensional vector space~$\DH$.

If $\deg(R) = \deg(\eps R') = 1$, then $R' \in \H[t]$ has a unique zero in $\H$,
let's say $r_1$. Moreover, $r_1 + \eps r_2 \in \DH$ is a zero of $R = \eps R'$
for any $r_2 \in \H$. In this case, the set of dual quaternion zeros of $R$ is a
four-dimensional affine subspace $A$ of~$\DH$.

We are interested in motion polynomial factorizations. Hence, the Study
condition $(t - r_1 - \eps r_2)(t - \Cj{r_1} - \eps \Cj{r_2}) \in \R[t]$ should
hold. It induces two linear constraints on $r_2$, namely $r_1\Cj{r_2} +
r_2\Cj{r_1} = 0$ and $r_2 + \Cj{r_2} = 0$. Hence, $\eps r_2$ lies in a
two-dimensional affine subspace $B$ of~$\eps\H \subset \DH$.

The intersection of $A$ and $C$ yields all common dual quaternion left/right
zeros of $N$ and $R$ and thus linear left/right factors of $P + \eps D$.
Intersecting with $B$ imposes the motion polynomial condition. If this
intersection is not empty (which is not guaranteed) we might even find new
factorizations for certain dual quaternion polynomials.
 
\begin{example}\label{ex:factor-darboux-motion}
  Consider the polynomial $P_4 = (t^2+1)(t-\qk) - \eps(\qi t^2 + (\qi + \qj) t +
  \qj)$, a special case of \cite[Example~5]{li15b}. The norm polynomial reads as
  $P_4\Cj{P_4} = (t^2+1)^3$. We can only choose the factor $N = t^2+1$ and
  compute $R = \rem(P_4,N) = -\eps((\qi + \qj) t - \qi + \qj)$. The set $A \cap
  B$ of its left zeros fulfilling the Study condition is $\{ -\qk + \eps
  (\lambda \qi + \mu \qj) \colon \lambda, \mu \in \R \}$. The set $C$ of left
  zeros of $N$ equals
  \begin{align*}
    C = 
    \{\alpha \qi + \beta \qj + \gamma \qk + \eps (\lambda \qi +
    \mu \qj) \colon \alpha, \beta, \gamma, \lambda, \mu \in \R \text{ and }
    \alpha^2+\beta^2+\gamma^2=1\}.
  \end{align*}
  Hence, $H = t + \qk - \eps(\lambda \qi + \mu \qj) \in \DH[t]$ is a left motion
  polynomial factor of $P_4$ for all $\lambda$, $\mu \in \R$. The right quotient
  $\rquo(P_4,H) = t^2 - 2\qk t - 1 + \eps ((\lambda \qi + \mu \qj - \qi) t - \mu
  \qi + \lambda \qj - \qi)$ can be factored by the generic factorization
  algorithm and we obtain $P_4 = H_1 H_2 H_3$ with
  \begin{align*}
    H_1 &= t + \qk - \eps \left(\lambda \qi + \mu \qj \right), \\
    H_2 &= t - \qk + \eps \bigl(( \lambda - \tfrac{1}{2} ) \qi + ( \mu + \tfrac{1}{2} ) \qj \bigr), \\
    H_3 &= t - \qk - \eps (\tfrac{1}{2} \qi + \tfrac{1}{2} \qj ),
  \end{align*}
  for all $\lambda$, $\mu \in \R$.
\end{example}

\begin{rmk}
  Our approach yields a systematic way to compute new factorizations for certain
  motion polynomials. Indeed, the leftmost linear factor in all factorizations
  of $P_4$ from Example~\ref{ex:factor-darboux-motion} in \cite{li15b} are
  quaternion factors while ours are proper dual quaternion factors.
\end{rmk}

\section{Future Research}

We have presented a new algorithm for factorization of split quaternion
polynomials. Its failure is equivalent to the non-factorizability of the
polynomial into linear factors. By investigation on the algorithm's geometric
behavior, we presented a procedure in order to factor all motions in the
hyperbolic plane represented by split quaternion polynomials.

Unfortunately, a test for factorizability of split quaternion polynomials using
our characterization is of factorial complexity in the polynomial degree as we
might need to run Algorithm~\ref{alg:non-generic-algorithm} for all possible
tuples of quadratic, real factors of the norm polynomial. An a priori
characterization as well as a numerically stable and robust version of
Algorithm~\ref{alg:non-generic-algorithm} is on our research agenda.

Our approach seems promising to find all possible factorizations into linear
factors of dual quaternion polynomials as well, a question which is also still
open yet. Further plans for future research include extensions of polynomial
factorization in the more general setup of Geometric Algebra. It admits a
conjugation, an important ingredient for the factorization algorithm but the
corresponding norm is not multiplicative which might cause problems. One
particularly interesting example is conformal geometric algebra for which
promising preliminary results do exist. The description in
\cite{dorst11b,dorst16} of the exponential in conformal geometric algebra allows
for a kinematic interpretation of factorization, similar to hyperbolic
kinematics (Section~\ref{sec:hyperbolic-motions-and-polynomials}).

\section*{Acknowledgment}

Daniel F. Scharler was supported by the Austrian Science Fund (FWF): P~31061
(The Algebra of Motions in 3-Space).

\bibliographystyle{elsarticle-harv}

\end{document}